\documentclass{amsart}

\usepackage{amsfonts,amssymb,amsmath,enumerate,verbatim,mathtools,tikz}
\usetikzlibrary{matrix}
\usetikzlibrary{arrows}

\usepackage[all]{xy}
\SelectTips{cm}{}

\newcommand{\Ass}{\operatorname{Ass}}

\newcommand\ext{\operatorname{Ext}}
\newcommand\tor{\operatorname{Tor}}
\newcommand{\Tor}[4]{\operatorname{Tor}_{#1}^{#2}(#3,#4)}
\newcommand{\Ext}[4]{\operatorname{Ext}^{#1}_{#2}(#3,#4)}

\newcommand\fm{{\mathfrak m}}

\newcommand\fq{{\mathfrak q}}

\newcommand\Hom{\operatorname{Hom}}

\newcommand\beg{\operatorname{beg}}

\newcommand{\im}{\operatorname{Im}}

\newcommand{\reg}{\operatorname{reg}_R}

\newcommand{\xra}{\xrightarrow}

\newtheorem{theorem}{Theorem}[section]

\newtheorem{lemma}[theorem]{Lemma}
\newtheorem{corollary}[theorem]{Corollary}

\theoremstyle{definition}
\newtheorem{example}[theorem]{Example}
\newtheorem*{definition}{Definition}

\theoremstyle{remark}
\newtheorem{remark}[theorem]{Remark}

\newtheorem{chunk}[theorem]{}

\numberwithin{equation}{section}

\begin{document}

\title{Regularity of Tor for weakly stable ideals}
\author{Katie Ansaldi, Nicholas Clarke, Luigi Ferraro}
\begin{abstract}
It is proved that if $I$ and $J$ are weakly stable ideals in a polynomial ring $R=k[x_1,\ldots,x_n]$, with $k$ a field, then the regularity of $\tor^R_i(R/I,R/J)$ has the expected upper bound. We also give a bound for the regularity of $\ext_R^i(R/I,R)$ for $I$ a weakly stable ideal.
\end{abstract}

\maketitle
\section*{Introduction}

Let $k$ be a field. Let $R=k[x_1,\ldots,x_n]$ be a graded polynomial ring over $k$ with $|x_i|=1$ for every $i$. Let $M$ and $N$ be finitely generated graded $R$-modules. In \cite{EisenbudTor} it is shown that if $\dim\tor_1^R(M,N)\leq 1$ then
\begin{equation}
\reg \tor_i^R(M,N) \leq \reg M +\reg N +i\quad\text{for\;every }i. \label{mainineq}
\end{equation}

In general this bound may not hold.  Indeed,  assume it holds for $M=N=R/I$ where $I$ is an homogeneous ideal in $R$ and set $T_1=\tor_1^R(R/I,R/I)$. It is clear that $T_1\cong I/I^2$; hence using the exact sequence
\[
0\rightarrow I^2\rightarrow I\rightarrow T_1\rightarrow0
\]
we deduce from \ref{lemses} that
\[
\reg I^2\leq\max\{\reg I, \reg T_1+1\}.
\]
Since $\reg R/I=\reg I-1$, it follows
\[
\reg I^2\leq2\reg I.
\]
Hence, every ideal not satisfying the previous inequality gives an example where (\ref{mainineq}) does not hold. There are many such examples; see for instance \cite{example}.

Although (\ref{mainineq}) does not hold in general, it is natural to look for classes of modules where the
bound holds without the dimension assumption.

We prove that if $I$ and $J$ are weakly stable ideals then
\[
\reg \tor_i^R(R/I,R/J) \leq \reg R/I +\reg R/J + i\quad\text{for\;every }i,
\]
see Theorem \ref{main} and see Section \ref{wsi} for the definition of weakly stable ideals.

The last section is concerned with the regularity of $\ext^i_R(R/I,R)$ with $I$ weakly stable ideal.

\section{Background}

Throughout the paper $R=k[x_1,\ldots,x_n]$, with $k$ a field, denotes a graded polynomial ring with $|x_i|=1$ for every $i$. Let $M$ and $N$ be finitely generated graded $R$-modules. We denote by $M_i$ the $i$-th graded component of $M$. The supremum and infimum of a graded module $M$ are defined as
\[
\sup M=\sup\{i\mid M_i\neq0\}
\]
\[
\inf M\;=\inf\{i\mid M_i\neq0\}.
\]

We define the graded $R$-module $M(-a)$ by $M(-a)_d=M_{a+d}$, the shift of $M$ up by $a$ degrees. Let $\fm$ denote the ideal $(x_1,\ldots,x_n)$. The  $\fm$-torsion functor on the category of graded $R$-modules is defined by
\[
\Gamma_{\fm}(M)=\{x\in M: \fm^t x=0\quad \text{for some }\,t\}.
\]
The $i$-th local cohomology module of $M$, denoted $H_{\fm}^i(M)$, is the $i$-th right derived functor of $\Gamma_{\fm}(\_)$ in the category of graded $R$-modules, and morphisms of degree 0.

We set $a_i(M)=\sup(H_{\fm}^i(M))$; by \cite[3.5.4]{CohenMacaulay} $a_i(M)$ is finite unless $H_{\fm}^i(M)=0$ where we set $a_i(M)= - \infty$. The Castelnuovo-Mumford regularity of $M$ is then
\begin{equation*}
\reg M=\sup_i \{a_i(M)+i\}.
\end{equation*}

Regularity can also be computed with a minimal graded free resolution
\[
\cdots\rightarrow F_2 \rightarrow F_1 \rightarrow F_0 \rightarrow 0
\]
of $M$. Recall that $F_i= \bigoplus_j R(-j)^{\beta_{ij}}$, so $\beta_{ij}$ is the number of copies of $R(-j)$ in position $i$ in the resolution. The number
\[
t_i(M)=\sup\{j: \beta_{ij} \neq 0\},
\]
is the largest degree of an element in the basis of $F_i$; it is easily seen that
\[
t_i(M)=\sup(\tor_i^R(M,k)).
\]
It is proved, for example in \cite[2.2]{Chardin}, that

\begin{equation*}
\reg M=\sup_i \{t_i(M)-i\}.
\end{equation*}

\begin{chunk}
If $M$ is an $R$-module then $\reg M(-a)=\reg M+a$ for any $a\in\mathbb{Z}$. This can be checked by computing the regularity with a free resolution.

If $M$ has finite length then $\reg M=\sup M$.  This follows by computing regularity with local cohomology.
\end{chunk}

\begin{chunk}\label{lemses}
Let
\[
0 \rightarrow L \rightarrow M \rightarrow N  \rightarrow 0
\]
be an exact sequence of graded $R$-modules. Then
\begin{enumerate}
\item $\reg M\leq \max \{\reg L,\reg N\}$
\item $\reg L\leq \max \{\reg M,\reg N+1\}$
\item $\reg N\leq \max \{\reg M,\reg L-1\}$.
\end{enumerate}
This follows from the induced long exact sequence in local cohomology.
\end{chunk}

The next lemma is a straightforward consequence of the previous inequalities.
\begin{lemma}\label{ExactFinite}
If $K\xra{f} M\xra{t} N\xra{g} C$ is an exact sequence of graded $R$-modules, $K$ and $C$ have finite length then
\[
\reg M\leq\max\{\reg K,\reg N,\reg C+1\}.
\]
\end{lemma}
\begin{proof}
The exact sequence induces exact sequences of $R$-modules
\[
0\rightarrow\im f\rightarrow M\rightarrow\im t\rightarrow0,\quad\quad\quad0\rightarrow\im t\rightarrow N\rightarrow\im g\rightarrow0.
\]
By \ref{lemses} these exact sequences give the following inequalities
\[
\reg M\leq\max\{\reg \im f, \reg \im t\}\quad\quad\quad\reg \im t\leq\max\{\reg N,\reg\im g+1\},
\]
and hence an inequality
\[
\reg M\leq\max\{\reg\im f,\reg N,\reg\im g+1\}.
\]
Since $K$ and $C$ have finite length $\reg\im f\leq\reg K$ and $\reg\im g\leq\reg C$.
\end{proof}

\begin{remark} \label{gamma}
Note that
\[
\reg M=\max\{\reg(\Gamma_\fm(M)), \reg(M/\Gamma_\fm(M))\}.
\]
This follows from the definition of regularity, since $H^0_\fm(M)=\Gamma_\fm(M)$.
\end{remark}

The following result is well-known.

\begin{lemma} \label{FiniteLength}
If $M$ has finite length then $\reg\tor^R_i(M,N)\leq\reg M+t_i(N)$.

In particular,
\[
\reg\tor_i^R(M,N)\leq\reg M+\reg N+i.
\]
\end{lemma}
\begin{proof}
Write $M=\bigoplus_{i=a}^b M_i$ with $a=\inf M$ and $b=\sup M$. We use induction on $b-a$. If $b=a$ then $M=k(-a)^m$, and therefore,
\begin{align*}
\reg\tor^R_i(M,N)&=\reg\tor^R_i(k(-a),N)\\
&=\reg\tor^R_i(k,N)(-a)\\
&=\reg\tor^R_i(k,N)+a\\
&=\reg\tor^R_i(k,N)+\reg M\\
&=t_i(N)+\reg(M).
\end{align*}

Now assume $b-a>0$. Denote by $M_{>a}$ the module $\bigoplus_{i=a+1}^bM_i$. The short exact sequence
\[
0\rightarrow M_{>a}\rightarrow M\rightarrow k(-a)^m\rightarrow0
\]
induces, for each $i$, an exact sequence
\[
\tor_i^R(M_{>a},N)\rightarrow\tor_i^R(M,N)\rightarrow\tor_i^R(k,N)^m(-a).
\]
By induction and Lemma \ref{ExactFinite}
\begin{align*}
\reg\tor^R_i(M,N)&\leq \max\{\reg\tor_i^R(M_{>a},N),\reg\tor_i^R(k,N)(-a)\}\\
                 &\leq\max\{\reg M_{>a}+t_i(N),a+t_i(N)\}\leq\reg M+t_i(N).
\end{align*}
The last assertion follows as $\reg N=\sup\{t_i(N)-i\}$.
\end{proof}

\section{Regularity of Tor for weakly stable ideals} \label{wsi}

We study weakly stable ideals. Let $I$ be a monomial ideal, for a monomial $ u \in I$ we let $m(u)$ be the maximum index of a variable appearing in $u$ and we let $l(u)$ be the highest power of $x_{m(u)}$ dividing $u$.

\begin{definition}
A monomial ideal $I$ is \emph{weakly stable} provided the following ``exchange property'' is satisfied; for any monomial $u\in I$ and for any $j<m(u)$ there exists a $k$ such that $x_j^ku/x_{m(u)}^{l(u)}\in I$.
\end{definition}

\begin{remark}
It is an easy exercise to prove that $I$ is weakly stable if and only if the ``exchange property'' is verified only for the generators of $I$.
\end{remark}

\begin{remark}\label{CavigliaLex}
There is also an algebraic characterization of weakly stable ideals. In \cite[4.1.5]{CavigliaThesis} Caviglia proved that a monomial ideal $I$ is weakly stable if and only if $\Ass I\subseteq\{(x_1,\ldots,x_t)\mid t=0,1,\ldots,n\}$.
\end{remark}

\begin{example}
 Let $I=(x_1^2,x_1x_2,x_1x_3,x_2^2)$. Clearly the 'exchange property' holds for $x_1^2$ and $x_2^2$. We have $m(x_1x_2)=2$ and $l(x_1x_2)=1$. For $j=1$ we take $k=1$ and we can see that $x_1x_1x_2/x_2$ is in $I$. The remaining generator is similar. The ideal $I$ is primary and the radical of $I$ is the ideal $(x_1,x_2)$.
\end{example}

\begin{remark}
\label{colon} If $I$ is a weakly stable ideal and $J$ is a monomial ideal, then $(I:J)$ is a weakly stable ideal; see \cite[4.1.4(2)]{CavigliaThesis}.
\end{remark}

\begin{lemma}\label{IPrime}
Suppose $I$ is a weakly stable ideal of $R$ and set
\[
I^\prime=\bigcup_{m=1}^\infty(I:x_n^m).
\]
Then $I^\prime$ is weakly stable and $\Gamma_\fm(R/I)=I^\prime/I$.
\end{lemma}
\begin{proof}
Notice that $I^\prime$ is the ideal of $R$ generated by the monomials obtained by setting $x_n=1$ in the generators of $I$.
First we show $I^\prime$ is weakly stable. We may assume $x_n|m$ for some $m \in G(I)$ where $G(I)$ denotes the set of minimal generators of $I$. Notice that if
\[
i=\max\{j\mid x_n^j\;\text{divides some }u\in G(I)\}
\]
then $I^\prime=(I:x_n^i)$ and this ideal is weakly stable by Remark \ref{colon}.

It is clear that $\Gamma_\fm(R/I)=\bigcup_i(I:\fm^i)/I$. We claim $\bigcup_i(I:\fm^i)=\bigcup_i(I:x_n^i)$. Take $f\in\bigcup_i(I:x_n^i)$ a monomial so that $fx_n^i\in I$ for some $i$. Since $I$ is weakly stable we can choose a $k$ such that $fx_j^k\in I$ for every $j$; hence, $f\in(I:\fm^{kn})$. The other inclusion is obvious.
\end{proof}

We are now ready to prove the main theorem.

\begin{theorem}\label{main}
If I and J are weakly stable ideals then
\[
\reg \tor_i^R(R/I,R/J) \leq \reg R/I +\reg R/J + i\quad\text{for\;every }i.
\]
\end{theorem}

\begin{proof}
Consider the following set

\begin{align*}
\mathfrak{F}=\{(I,J)\mid \;& I,J\, \text{are weakly stable ideals and}\\
&\reg\tor_i^R(R/I,R/J)>\reg R/I+\reg R/J+i\text{ for some } i\}.
\end{align*}
This set is partially ordered as follows: $(I,J)\leq(I^\prime, J^\prime)$ if $I\subseteq I^\prime$ and $J\subseteq J^\prime$.
Assume that $\mathfrak{F}\neq\emptyset$, we seek a contradiction. Since $R$ is noetherian there exists a maximal element $(I,J)$.

We may assume $x_n|m$ for some $m\in G(I)\cup G(J)$. Otherwise, we let $S=k[x_1,\ldots,x_{n-1}]$, then
\[
\tor_i^R(R/I,R/J)\cong\tor^S_i(S/I\cap S,S/J\cap S)\otimes_SR\quad\text{for\;every }i.
\]
Regularity does not change under faithfully flat extensions; hence it is enough to prove the theorem for $S$.
Moreover, as Tor is symmetric we can assume that $x_n|m$ for some $m\in G(I)$.

By Lemma \ref{IPrime}, $\Gamma_\fm(R/I)=I^\prime/I$, so there is an exact sequence
\[
0\rightarrow \Gamma_\fm(R/I)\rightarrow R/I\rightarrow R/I^\prime\rightarrow0
\]
which induces, for each $i$, an exact sequence
\begin{align*}
\cdots & \rightarrow\tor_i^R(\Gamma_\fm(R/I),R/J)\rightarrow\tor_i^R(R/I,R/J)\rightarrow\tor^R_i(R/I^\prime,R/J)\rightarrow\\
&\rightarrow\tor_{i-1}^R(\Gamma_\fm(R/I),R/J).
\end{align*}
The outside terms have finite length, since $\Gamma_\fm(R/I)$ has finite length, and therefore by Lemma \ref{ExactFinite}
\begin{align*}
\reg\tor_i^R(R/I,R/J)
\leq\max\{ &\reg\tor^R_i(\Gamma_\fm(R/I),R/J),\\
& \reg\tor_i^R(R/I^\prime,R/J),\\
&\reg\tor_{i-1}^R(\Gamma_\fm(R/I),R/J)+1\}.
\end{align*}
We examine the terms on the right hand side. By \ref{FiniteLength} and \ref{gamma} we have
\begin{align*}
\reg\tor_i^R(\Gamma_\fm(R/I),R/J) &\leq\reg \Gamma_\fm(R/I)+\reg R/J+i\\
                               &\leq\reg R/I+\reg R/J+i
\end{align*}
and
\begin{align*}
\reg\tor_{i-1}^R(\Gamma_\fm(R/I),R/J)+1 &\leq\reg \Gamma_\fm(R/I)+\reg R/J+i-1+1\\
                                     &\leq\reg R/I+\reg R/J+i.
\end{align*}
By \ref{IPrime} we know $I^\prime$ is weakly stable. As $I \subsetneq I^\prime$ and the pair $(I,J)$ is maximal in $\mathfrak{F}$
\begin{align*}
\reg\tor_i^R(R/I^\prime,R/J) &\leq\reg R/I^\prime+\reg R/J+i\\
                             &\leq\reg R/I+\reg R/J+i.
\end{align*}
The final inequality follows by \ref{gamma} since
\[
R/I^\prime=\frac{R/I}{\Gamma_\fm(R/I)}.
\]
Putting all these inequalities together gives us
\[
\reg\tor_i^R(R/I,R/J)\leq\reg R/I+\reg R/J+i\quad\text{for\;every }i.
\]
This is a contradiction since $(I,J)\in\mathfrak{F}$.
\end{proof}
\begin{remark}
The inequality in Theorem \ref{main} is useful because Caviglia gives a formula for the regularity of weakly stable ideals (see \cite[4.1.10]{CavigliaThesis}).
\end{remark}
\section{Regularity of Ext for weakly stable ideals}
Let $M$ be an $R$-module of dimension $d$.
Regularity of $\ext^i_R(M,R)$ was studied, for example, in \cite{HoaHyry}; here we study it in the case $M=R/I$ with $I$ a weakly stable ideal.

\begin{lemma}
Let $R=k[x_1,\ldots,x_n]$. If $M$ is an $R$-module of finite length then $\Ext iRMR=0$ for $i<n$ and
\[
\reg\Ext nRMR=-n-\inf M.
\]
\end{lemma}
\begin{samepage}
\begin{proof}
By graded local duality, see \cite[Theorem 3.6.19]{CohenMacaulay}, there is the following isomorphism:
\[
\Hom_{R}(H^i_\fm(M),E)\cong\ext_{R}^{n-i}(M,R(-n))\cong\ext_{R}^{n-i}(M,R)(-n),
\]
where $E$ is the injective hull of $k$. Since $M$ has finite length all the local cohomology modules are zero for $i>0$ and $H^0_\fm(M)=M$. This gives $\ext^i_{R}(M,R)=0$ for $i<n$. The last assertion follows since
\[
\reg\Hom_{R}(M,E)=\sup\Hom_{R}(M,E)=-\inf M. \qedhere
\]

\end{proof}
\end{samepage}

\begin{theorem} \label{main2}
If I is a weakly stable ideal then
\[
\reg\Ext iR{R/I}R\leq-i\quad\text{for\;every }i.
\]
\end{theorem}
\begin{proof}
Set
\begin{align*}
\mathfrak{F}=\{I\mid \;& I\, \text{is a weakly stable ideal such that}\\
&\reg\ext^i_R(R/I,R)>-i\text{ for some } i\}.
\end{align*}
Notice that $\mathfrak{F}$ is partially ordered by inclusion of ideals.

The theorem asserts that $\mathfrak{F}$ is empty, so we assume it is not and argue by contradiction. Since $R$ is noetherian there exists $I\in\mathfrak{F}$ a maximal element.

We may assume $x_n|m$ for some $m\in G(I)$; otherwise, let $S=k[x_1,\ldots,x_{n-1}]$. Then
\[
\ext_R^i(R/I,R)\cong\ext^i_S(S/I\cap S,S)\otimes_SR\quad\text{for\;every }i
\]
and regularity does not change under faithfully flat extensions. Hence, it is enough to prove the theorem for $S$.

By Lemma \ref{IPrime} we have $\Gamma_\fm(R/I)=I^\prime/I$, with $I^\prime$ weakly stable and $I\subsetneqq I^\prime$ so by maximality the assertion holds for $I^\prime$.

The short exact sequence
\[
0\rightarrow \Gamma_\fm(R/I)\rightarrow R/I\rightarrow R/I^\prime\rightarrow0
\]
induces, for each $i$, an exact sequence
\begin{align*}
&\hphantom{\rightarrow}\Ext {i-1}R{\Gamma_\fm(R/I)}R\rightarrow\\
\rightarrow \Ext iR{R/I^\prime}R\rightarrow\Ext iR{R/I}R&\rightarrow\Ext iR{\Gamma_\fm(R/I)}R.
\end{align*}
If $i<n$ then, since $\Gamma_\fm(R/I)$ has finite length, the outside terms are zero, giving the isomorphism $\Ext iR{R/I^\prime}R\cong\Ext iR{R/I}R$;  hence, the assertion holds for $I$ and $i<n$. If $i=n$ we get a short exact sequence
\[
0\rightarrow\Ext nR{R/I^\prime}R\rightarrow\Ext nR{R/I}R\rightarrow\Ext nR{\Gamma_\fm(R/I)}R\rightarrow 0.
\]
As the bound holds for $I^\prime$ and $\Gamma_\fm(R/I)$ has finite length
\begin{align*}
\reg\Ext nR{R/I}R\leq\max\{&\reg\Ext nR{R/I^\prime}R,\\
&\reg\Ext nR{\Gamma_\fm(R/I)}R\}\leq-n\\
\end{align*}
Thus the bound holds for $I$ and for every $i$; this is the desired contradiction.
\end{proof}

The previous result can be also deduced from a result of Hoa and Hyry. They prove (see \cite[Proposition 22]{HoaHyry}) that if $M$ is a sequentially Cohen-Macaulay module (see \cite{SCM} for the definition) then
\[
\reg(\ext_R^i(M,R))\leq -i-\inf M\quad\text{for\;every }i.
\]
Caviglia and Sbarra proved (see \cite[1.10]{CavigliaSbarra}) that if $I$ is weakly stable then $R/I$ is sequentially CM, hence Hoa and Hyry's inequality reduces to
\[
\reg(\ext^i_R(R/I,R))\leq -i\quad\text{for\;every }i.
\]

\section*{Acknowledgements}

Most of this work was carried out at the PRAGMATIC workshop at the University of Catania in the summer of $2014$. Thanks go to the organizers of this conference, Alfio Ragusa, Giuseppe Zappal\`{a}, Francesco Russo and Salvatore Giuffrida. We would also like to thank the lecturers Aldo Conca, Srikanth Iyengar and Anurag Singh. Special thanks go to Aldo Conca for suggesting the project and for many informative discussions and to Srikanth Iyengar for useful comments.

\end{document}